\documentclass[MSNbibl,number,citesort,dvips]{arxbj}
\usepackage{upgreek}
\usepackage{graphicx}

\aid{0}
\volume{20}
\issue{3}
\pubyear{2014}
\firstpage{1600}
\lastpage{1619}
\doi{10.3150/13-BEJ534} 

\makeatletter
\def\pi{\uppi}
\newtheorem{lemma}{Lemma}
\newtheorem{theorem}{Theorem}
\newtheorem{proposition}{Proposition}
\newtheorem{corollary}{Corollary}
\newproclaim{definition}{Definition}
\newproclaim{conjecture}{Conjecture}
\newproclaim{assumption}{Assumption}
\newcommand{\R}{\mathbb{R}}
\newcommand{\Z}{\mathbb{Z}}

\newcommand{\Prob}{\mathbb{P}}
\newcommand{\Q}{\mathbb{Q}}
\newcommand{\Expec}{\mathbb{E}}
\newcommand{\Var}{\operatorname{\mathbb{V}ar}}
\newcommand{\Cov}{\operatorname{\mathbb{C}ov}}
\makeatother

\begin{document}
\begin{frontmatter}

\title{On asymptotic constants in the theory of extremes for Gaussian processes}

\runtitle{On asymptotic constants in the theory of extremes for
Gaussian processes}

\begin{aug}
\author[a]{\inits{A.B.}\fnms{A.B.} \snm{Dieker}\corref{}\thanksref{a}\ead[label=e1]{ton.dieker@isye.gatech.edu}}
\and
\author[b]{\inits{B.}\fnms{B.} \snm{Yakir}\thanksref{b}\ead[label=e2]{msby@mscc.huji.ac.il}}
\runauthor{A.B. Dieker and B. Yakir} 
\address[a]{H. Milton Stewart School of Industrial and Systems Engineering,
Georgia Institute of Technology, Atlanta, GA 30332, USA.
\printead{e1}}
\address[b]{Department of Statistics, Hebrew University of Jerusalem,
Mount Scopus, Jerusalem 91905, Israel.
\printead{e2}}
\end{aug}

\received{\smonth{8} \syear{2012}}
\revised{\smonth{3} \syear{2013}}

%
\begin{abstract}
This paper gives a new representation of Pickands' constants, which
arise in the study
of extremes for a variety of Gaussian processes.
Using this representation, we resolve the long-standing problem of
devising a
reliable algorithm for estimating these constants.
A detailed error analysis illustrates the strength of our approach.
\end{abstract}

%

\end{frontmatter}

\section{Introduction}
Gaussian processes and fields have emerged as a versatile yet relatively
tractable class of models
for random phenomena.
Gaussian processes have been applied fruitfully to risk theory, statistics,
machine learning, and biology, while
Gaussian fields have been applied to neuroimaging, astrophysics, oceanography,
as well as to other fields.
Extremes and level sets are particularly important in these applications
(Aza{\"{\i}}s and Wschebor \cite{azaisMR2478201}).
New applications and theoretical developments
continually revive the interest in Gaussian processes, see, for
instance, Meka \cite{meka:PTAS}.

Although the understanding of Gaussian processes and fields has advanced
steadily over the past decades, a variety of results related to extremes
(tail asymptotics, extreme value theorems, laws of iterated logarithm)
are only
``explicit'' up to certain constants.
These constants are referred to as Pickands' constants after their discoverer
(Pickands, III \cite{pickands:asymptotic1969}).
It is believed that these constants
may never be calculated (Adler \cite{adler:gaussianprocesses1990}).

These constants have remained so elusive that devising an estimation algorithm
with certain performance guarantees has remained outside the scope of
current methodology (D{\c{e}}bicki and Mandjes \cite{MR2834197}).
The current paper resolves this open problem for the classical
Pickands' constants.
Our main tool is a new representation for Pickands' constant,
which expresses the constant as the expected
value of a random variable with low variance and therefore it is
suitable for simulation.
Our approach also gives rise to a number of new questions,
which could lead to further improvement of our simulation algorithm or
its underlying theoretical foundation.
We expect that our methodology carries through for all of Pickands'
constants, not
only for the classical ones discussed here.

Several different representations of Pickands' constants are known,
typically arising from various methodologies for studying extremes of Gaussian
processes.
H{\"u}sler \cite{husler:extremes1999} uses triangular arrays to interpret
Pickands' constant as a clustering index.
Albin and Choi \cite{MR2685014} have recently rediscovered H\"usler's
representation.
For sufficiently smooth Gaussian processes, various level-crossing
tools can be
exploited
(Aza{\"{\i}}s and
Wschebor \cite{azaiswschebor:maxfield2005},
Kobelkov \cite{kobelkov:ruin2005}).
Yet another representation is found when a sojourn approach is
taken (Berman~\cite{berman:sojournsextremes1992}).
Aldous \cite{aldous:poissonclumping1989} explains
various connections heuristically and also gives intuition
behind other fundamental results in extreme-value theory.
We also mention Chapter~12 in Leadbetter \textit{et~al.} \cite{leadbetter:extremes1983},
who use methods different from those of Pickands but arrive at the same
representation.

The approach advocated in the current paper is inspired by a method
which has
been applied successfully
in various statistical settings, see Siegmund \textit{et~al.} \cite{MR2747524} and references therein.
This method relies on a certain change-of-measure argument,
which results in asymptotic expressions with a term of the form
$\Expec(M/S)$, where $M$ and $S$ are
supremum-type and sum-type (or integral-type) functionals, respectively.
This methodology can also be applied directly to study extremes of Gaussian
processes,
in which case it yields a new method for establishing tail asymptotics.
This will be pursued elsewhere.

Throughout this paper, we let $B=\{B_t\dvt t\in\R\}$ be a standard fractional
Brownian motion
with Hurst index $\alpha/2\in(0,1]$, that is, a centered Gaussian
process for
which
\[
\Cov(B_s,B_t) = \tfrac12 \bigl[|s|^\alpha+
|t|^\alpha- |t-s|^\alpha \bigr].
\]
Note that has stationary increments and variance function $\Var(B_t) =
|t|^{\alpha}$.
The process $\{Z_t\}$ defined through $Z_t = \sqrt{2} B_t - |t|^\alpha$
plays a key role in this paper.
This stochastic process plays a fundamental role in the stochastic
calculus for
fractional
Brownian motion (Bender and Parczewski \cite{MR2668907}).
The ``classical'' definition of Pickands' constant $\mathcal H_\alpha$ is
%
%
\begin{equation}
\label{eq:defpickands} \mathcal H_\alpha= \lim_{T\to\infty}
\frac1T \Expec \Bigl[\sup_{t\in[0,T]} \mathrm{e}^{Z_t} \Bigr].
\end{equation}

Current understanding of $\mathcal H_\alpha$ and related constants is quite
limited.
It is known that $\mathcal H_{1}=1$ and that $\mathcal H_{2} = 1/\sqrt
\pi$
(Bickel and
Rosenblatt \cite{bickelrosenblatt:deviations1973},
Piterbarg \cite{piterbarg:asymptoticmethods1996}),
and that $\mathcal H_\alpha$ is continuous as a function of $\alpha$
(D{\c{e}}bicki \cite{debicki:remarks2002}).
Most existing work focuses on obtaining sharp bounds for these constants
(Aldous \cite{aldous:poissonclumping1989},
D{\c{e}}bicki \cite{debicki:ruinprob2002},
D{\c{e}}bicki and Kisowski \cite{MR2458013},
D{\c{e}}bicki
\textit{et~al.} \cite{debicki:simulation2003},
Shao \cite{shao:bounds1996},
Harper \cite{harper2010}).
Previous work on estimating Pickands' constant through simulation has
yielded contradictory results
(Burnecki and Michna \cite{MR1944151},
Michna \cite{michna:tailprob1999}).

The next theorem forms the basis for our approach to estimate $\mathcal
H_\alpha$.
Note that the theorem expresses $\mathcal H_\alpha$ in the form $\Expec(M/S)$.
A different but related representation is given in
Proposition~\ref{prop:discreteanalog} below, and
we give yet another representation in Proposition~\ref{prop:albintyperepr}.

\begin{theorem}
\label{thm:newrepr}
We have
\[
\mathcal H_\alpha= \Expec \biggl[\frac{\sup_{t\in\R}\mathrm{e}^{Z_t}}{\int_{-\infty}^\infty
\mathrm{e}^{Z_t}\,\mathrm{d}t} \biggr].
\]
\end{theorem}

The representation $\Expec(M/S)$ is well-suited for estimating Pickands'
constant by simulation.
Although both $M$ and $S$ are finite random variables with infinite mean,
we provide theoretical evidence that their ratio has low variance
and our empirical results show that this representation is suitable for
simulation.

This paper is organized as follows.
Section~\ref{sec:representations} establishes two results which
together yield
Theorem~\ref{thm:newrepr}.\vadjust{\goodbreak}
In Section~\ref{sec:auxiliary}, we state an auxiliary result that plays a
key role in several of the proofs in this paper.
Section~\ref{sec:estimation} gives an error analysis when $\Expec(M/S)$ is
approximated by a related quantity that can be simulated on a computer.
In Section~\ref{sec:numexp}, we carry out simulation experiments to estimate
Pickands' constant.
Some proofs are deferred to Appendix~\ref{appendix}, and a table with our simulation
results is included as Appendix~\ref{sec:simulatedvalues}.

\section{Representations}
\label{sec:representations}
This section is devoted to connections between Pickands' classical
representation
and our new representation, thus establishing Theorem~\ref{thm:newrepr}.
We also informally argue why our new representation is superior from
the point
of view of estimation. This is explored further in the next section.

The following well-known change-of-measure lemma forms the basis for
our results.

\begin{lemma}
\label{lem:changeofmeasure}
Fix $t\in\R$, and set $Z^{(t)} = \{\sqrt{2} B_s -|s-t|^{2H}\dvt s\in\R\}$.
For an arbitrary measurable functional $F$ on $\R^\R$, we have
\[
\Expec \mathrm{e}^{Z_t} F(Z) = \Expec F\bigl(|t|^{2H} + Z^{(t)}
\bigr).
\]
When the functional $F$ is moreover
translation-invariant (invariant under
addition of a constant function), we have
\[
\Expec \mathrm{e}^{Z_t} F(Z) = \Expec F(\theta_t Z),
\]
where the shift $\theta_t$ is defined through $(\theta_t Z)_s = Z_{s-t}$.
\end{lemma}
\begin{pf}
Set $\Q(A) = \Expec[\mathrm{e}^{Z_t} 1_A]$, and write $\Expec^\Q$ for the
expectation operator with respect to $\Q$.
Select an integer $k$ and $s_1<s_2<\cdots<s_k$.
We show that $(Z_{s_1},\ldots,Z_{s_k})$ under $\Q$ has the same
distribution as $(|t|^{2H}+Z^{(t)}_{s_1},\ldots
,|t|^{2H}+Z^{(t)}_{s_k})$ under $\Prob$,
by comparing generating functions: for any $\beta_1,\ldots,\beta_k\in\R$,
\begin{eqnarray*}
\log\Expec^\Q\exp \biggl(\sum
_i \beta_i Z_{s_i} \biggr)
&= &-|t|^{2H} -\sum_i
\beta_i |s_i|^{2H} +\Var \biggl[B_t
+ \sum_i \beta_i B_{s_i}
\biggr]
\\
&= &\sum_i 2\beta_i
\Cov(B_t, B_{s_i})-\sum_i
\beta_i |s_i|^{2H} +\Var \biggl[\sum
_i \beta_i B_{s_i} \biggr]
\\
&=& \sum_i \beta_i
\bigl[|t|^{2H}-|s_i-t|^{2H} \bigr]+\Var \biggl[\sum
_i \beta_i B_{s_i} \biggr]
\\
&=& \sum_i\beta_i |t|^{2H}
+\Expec \biggl[\sum_i\beta_i
Z^{(t)}_{s_i} \biggr]+ \frac12 \Var \biggl[\sum
_i\beta_i Z^{(t)}_{s_i}
\biggr]
\\
&=& \Expec \biggl[\sum_i\beta_i
\bigl(|t|^{2H}+Z^{(t)}_{s_i}\bigr) \biggr]+ \frac12 \Var
\biggl[\sum_i\beta_i
\bigl(|t|^{2H} + Z^{(t)}_{s_i}\bigr) \biggr].
\end{eqnarray*}
The first claim of the lemma then immediately follows from the Cram\'
er--Wold device.

Alternatively, one could carefully define a space on which the
distribution of $Z$ becomes a Gaussian measure
and then note that the claim follows from the Cameron--Martin formula;
see Bogachev \cite{bogachev:gaussianmeasures1998}, Proposition~2.4.2 and Dieker \cite
{dieker:condlim2005}
for key ingredients for this approach.

When the functional $F$ is
translation-invariant, we conclude that
\[
\Expec^\Q F(Z) = \Expec F\bigl(|t|^{2H} + Z^{(t)}
\bigr) = \Expec F\bigl(Z^{(t)}-\sqrt 2B_t\bigr) = \Expec F(
\theta_t Z),
\]
and this proves the second claim in the lemma.
\end{pf}

The next corollary readily implies subadditivity of
$\Expec [\sup_{0\le t\le T} \mathrm{e}^{Z_t} ]$ as a function of $T$,
a well-known fact that immediately yields the existence of the
limit in (\ref{eq:defpickands}).
Evidently, we must work under the usual separability conditions, which
ensure that the supremum functional is measurable.

\begin{corollary}
For any $a<b$, we have
\[
\Expec \Bigl[\sup_{a\le t\le b} \mathrm{e}^{Z_t} \Bigr] = \Expec \Bigl[
\sup_{0\le t\le b-a} \mathrm{e}^{Z_t} \Bigr], \qquad\Expec \biggl[\int
_a^b \mathrm{e}^{Z_t}\,\mathrm{d}t \biggr] = \Expec \biggl[
\int_0^{b-a} \mathrm{e}^{Z_t}\,\mathrm{d}t \biggr].
\]
\end{corollary}
\begin{pf}
Applying Lemma~\ref{lem:changeofmeasure} for $t=a$ to the
translation-invariant functionals $F$
given by
$
F(z) = \sup_{a\le s\le b} \mathrm{e}^{z_s-z_a}
$
and $F(z) = \int_a^b \mathrm{e}^{z_t-z_a}\,\mathrm{d}t$
yields the claims. (The second claim is also immediate from $\Expec \mathrm{e}^{Z_t}=1$.)
\end{pf}

\begin{corollary}
\label{cor:relationclassicalnew}
For $T>0$, we have
%
%
\begin{equation}
\label{eq:relationclassicalnew} \frac{1}{T}\Expec \Bigl[\sup
_{0\leq t \leq
T}\mathrm{e}^{Z_t} \Bigr]= \int_{0}^{1}
\Expec \biggl[\frac{\sup_{-u T\leq s \leq
(1-u)T}\mathrm{e}^{Z_s}}{\int_{-uT}^{(1-u)T}\mathrm{e}^{Z_s}\,\mathrm{d}s} \biggr] \,\mathrm{d}u.
\end{equation}
%
\end{corollary}
\begin{pf}
Applying Lemma~\ref{lem:changeofmeasure} to the translation-invariant
functional $F$ given by
\[
F(z) = \frac{\sup_{t\in[0,T]} \mathrm{e}^{z_t}}{\int_{0}^T \mathrm{e}^{z_u} \,\mathrm{d}u}
\]
yields that, for any $T>0$,
\begin{eqnarray*}
\frac{1}{T}\Expec \Bigl[\sup_{0\leq t \leq
T}\mathrm{e}^{Z_t}
\Bigr]&=&\frac{1}{T}\int_{0}^{T}\Expec
\biggl[\mathrm{e}^{Z_t} \times \frac{\sup_{0\leq s \leq
T}\mathrm{e}^{Z_s}}{\int_{0}^{T}\mathrm{e}^{Z_s}\,\mathrm{d}s} \biggr] \,\mathrm{d}t
\\
&=& \frac{1}{T}\int_{0}^{T}\Expec \biggl[
\frac{\sup_{-t\leq s \leq
T-t}\mathrm{e}^{Z_s}}{\int_{-t}^{T-t}\mathrm{e}^{Z_s}\,\mathrm{d}s} \biggr] \,\mathrm{d}t,
\end{eqnarray*}
and the statement of the lemma follows after a change of variable.
\end{pf}

The left-hand side of the identity (\ref{eq:relationclassicalnew})
converges to
$\mathcal H_\alpha$
by definition. The next proposition shows that the right-hand side of
(\ref{eq:relationclassicalnew})
converges to our new representation, thereby proving Theorem~\ref{thm:newrepr}.
The proof of the proposition itself is deferred to Appendix~\ref{appendix}.

\begin{proposition}
\label{prop:UIMT}
For any $u\in(0,1)$, we have
\[
\lim_{T\to\infty} \Expec \biggl[\frac{\sup_{-u T\leq s \leq
(1-u)T}\mathrm{e}^{Z_s}}{\int_{-uT}^{(1-u)T}\mathrm{e}^{Z_s}\,\mathrm{d}s} \biggr] = \Expec
\biggl[\frac{\sup_{t\in\R}\mathrm{e}^{Z_t}}{\int_{-\infty}^\infty
\mathrm{e}^{Z_t}\,\mathrm{d}t} \biggr]<\infty.
\]
Moreover,
\[
\mathcal H_\alpha= \lim_{T\to\infty} \int
_{0}^{1}\Expec \biggl[\frac{\sup_{-u T\leq s \leq
(1-u)T}\mathrm{e}^{Z_s}}{\int_{-uT}^{(1-u)T}\mathrm{e}^{Z_s}\,\mathrm{d}s} \biggr]
\,\mathrm{d}u= \Expec \biggl[\frac{\sup_{t\in\R}\mathrm{e}^{Z_t}}{\int_{-\infty}^\infty
\mathrm{e}^{Z_t}\,\mathrm{d}t} \biggr].
\]
\end{proposition}

Apart from establishing Theorem~\ref{thm:newrepr}, this proposition
gives two ways of approximating $\mathcal H_\alpha$.
The speed at which the prelimits tend to $\mathcal H_\alpha$ is
different for
these two representations.
For the second ``integral'' representation, which is the classical representation
in view of Corollary~\ref{cor:relationclassicalnew}, the speed of
convergence to
$\mathcal H_\alpha$
can be expected to be slow.
Indeed, it is known to be of order $1/\sqrt{T}$ in the Brownian motion case
(e.g., D{\c{e}}bicki and Kisowski \cite{MR2458013}).
This is in stark contrast with the speed of convergence
in the first representation (e.g., for $u=1/2$), as analyzed in the next
section.
Our study shows that the slow convergence speed in the classical
definition is due to values of $u$ close to the endpoints of the
integration interval $[0,1]$ in the right-hand side of
(\ref{eq:relationclassicalnew}).

It is instructive
to compare our new representation of $\mathcal H_\alpha$
with the classical representation of Pickands' constant
through a discussion of variances.
Note that $\Expec \mathrm{e}^{Z_s}=1$, $\Var \mathrm{e}^{Z_s} = \mathrm{e}^{\Var(Z_s)}-\mathrm{e}^{-\Var(Z_s)}$,
so that the variance blows up as $s$ grows large. As a result, one can
expect that $\sup_{0\le t\le T} \mathrm{e}^{Z_t}$ has high variance for large $T$.
Moreover, significant contributions to its expectation come from values
of $t$
close to $T$.
These two observations explain why it is hard to reliably estimate
Pickands' constant from the classical definition.

Our new representation does not have these drawbacks.
Let us focus on the special case $\alpha=2$,
for which it is known that $\mathcal H_2=1/\sqrt{\pi}$.
Writing $N$ for a standard normal random variable, we obtain that
\[
\mathcal H_2 = \Expec \biggl[\frac{\sup_{t\in\R} \mathrm{e}^{\sqrt{2} t N - t^2}} {
\int_{t\in\R} \mathrm{e}^{\sqrt{2} t N - t^2}\,\mathrm{d}t} \biggr] = \Expec
\biggl[\frac{\sup_{t\in\R} \mathrm{e}^{- (t-N/\sqrt{2})^2}} {
\int_{t\in\R} \mathrm{e}^{- (t-N/\sqrt{2})^2}\,\mathrm{d}t} \biggr] = \frac1{\int_{\R}
\mathrm{e}^{-t^2} \,\mathrm{d}t} = \frac1{\sqrt{\pi}}.
\]
It follows from this calculation that $M/S$ has zero
variance for $\alpha=2$, so we
can expect it to have very low variance for values of $\alpha$ close to 2.

We next present an alternative representation for $\mathcal H_\alpha$
in the spirit of Theorem~\ref{thm:newrepr}.
The proof of Corollary~\ref{cor:relationclassicalnew} shows that for
any locally finite measure $\mu$,
%
%
\begin{equation}
\label{eq:reprcountingmeasure} \frac1T \Expec \Bigl[\sup_{0\le t\le T}
\mathrm{e}^{Z_t} \Bigr] = \int_{0}^{1}\Expec
\biggl[\frac{\sup_{-uT\leq s \leq
(1-u)T}\mathrm{e}^{Z_s}}{\int_{-uT}^{(1-u)T}\mathrm{e}^{Z_s}\mu(\mathrm{d}s)} \biggr] \mu^{(T)}(\mathrm{d}u),
\end{equation}
where $\mu^{(T)}(\mathrm{d}u) = \mu(T \,\mathrm{d}u)/T$.
Of particular interest is the case where $\mu$ is the counting measure
on $\eta\Z$.
Then $\mu^{(T)}$ converges weakly to $\textrm{Leb}/\eta$, where $\textrm
{Leb}$ stands for
Lebesgue measure. In view of this observation, the following analog of
Proposition~\ref{prop:UIMT}
is natural. The proof is given in Appendix~\ref{appendix}.
%
\begin{proposition}
\label{prop:discreteanalog}
For any $\eta>0$, we have
\[
\mathcal H_\alpha= \Expec \biggl[\frac{\sup_{t\in\R} \mathrm{e}^{Z_t}}{\eta\sum_{k\in\Z} \mathrm{e}^{Z_{k \eta}}} \biggr].
\]
\end{proposition}
This identity is particularly noteworthy since the integral in the
denominator in
the representation of Theorem~\ref{thm:newrepr} can apparently be
replaced with an approximating sum.
For $\alpha=2$, this means that for any $\eta>0$,
\[
\int_{\R} \frac{\mathrm{d}y}{\sum_{k\in\Z} \mathrm{e}^{ky\eta^2-k^2\eta^2}} = 2.
\]
We have not been able to verify this intriguing equality directly, but
numerical experiments
suggest that this identity indeed holds.

We conclude this section with two further related results. For $\eta
>0$, define the ``discretized'' Pickands constant through
\[
\mathcal H_\alpha^\eta= \lim_{T\to\infty} \frac1T
\Expec \Bigl[\sup_{k\in\Z: 0\le k\eta\le T} \mathrm{e}^{Z_{k\eta}} \Bigr].
\]
The proof of the next proposition requires discrete analogs of
Corollary~\ref{cor:relationclassicalnew} and
Proposition~\ref{prop:UIMT}, with suprema taken over a grid and
integrals replaced by sums (for the first equality).
The proof is omitted since it follows the proofs of these results verbatim.
%
\begin{proposition}
\label{prop:discretepickands}
For any $\eta>0$, we have
\[
\mathcal H_\alpha^\eta=\Expec \biggl[\frac{\sup_{k\in\Z} \mathrm{e}^{Z_{k\eta
}}}{\eta\sum_{k\in\Z} \mathrm{e}^{Z_{k \eta}}}
\biggr] =\Expec \biggl[\frac{\sup_{k\in\Z} \mathrm{e}^{Z_{k\eta}}}{\int_{-\infty}^\infty
\mathrm{e}^{Z_{t}}\,\mathrm{d}t} \biggr].
\]
\end{proposition}
The second representation for $\mathcal H_\alpha^\eta$ in this
proposition immediately shows that $\mathcal H_\alpha= \lim_{\eta
\downarrow0} \mathcal H_\alpha^\eta$
by the monotone convergence theorem and sample path continuity.

A different application of Lemma~\ref{lem:changeofmeasure} yields further
representations for $\mathcal H_\alpha$ and $\mathcal H_\alpha^\eta$.
Let $F_t$ be the indicator of the event that the supremum of its
(sample path) argument occurs at $t$.
Since $E[F_t(Z)F_s(Z)] = 0$ for all $s \neq t$, we have
\[
\Expec \Bigl[\sup_{k\in\Z: 0\le k\eta\le T} \mathrm{e}^{Z_{k\eta}} \Bigr] = \sum
_{\ell= 0}^{\lfloor T/\eta\rfloor} \Expec \bigl[\mathrm{e}^{Z_{\ell\eta}}F_{\ell
\eta}(Z)
\bigr] = \sum_{\ell=0}^{\lfloor T/\eta\rfloor} \Prob \Bigl(\sup
_{k\in\Z: -\ell
\le k \le T/\eta-\ell} Z_{k\eta} = 0 \Bigr),
\]
where we use Lemma~\ref{lem:changeofmeasure} to obtain the last equality.
This can be written as
\[
\frac{1}{T}\Expec \Bigl[\sup_{k\in\Z: 0\le k\eta\le T} \mathrm{e}^{Z_{k\eta
}}
\Bigr] = \int_0^1 \Prob \Bigl( \sup
_{k\in\Z: -uT\le k\eta\le(1-u)T} Z_{k\eta} = 0 \Bigr)\mu^{(T)} (\mathrm{d}u),
\]
where, as before, $\mu^{(T)}(\mathrm{d}u) = \mu(T\,\mathrm{d}u)/T$ and $\mu$ is the
counting measure on $\eta\Z$.
Note the similarity with (\ref{eq:reprcountingmeasure}).
Taking the limit as $T\to\infty$ requires verifications similar to
those in the proof of Proposition~\ref{prop:discreteanalog};
the details are given in Appendix~\ref{appendix}.
The resulting representation is a two-sided version of the H\"
usler--Albin--Choi representation (Albin and Choi \cite{MR2685014},
H{\"u}sler \cite{husler:extremes1999}),
and appears to be new.

\begin{proposition}
\label{prop:albintyperepr}
For $\eta>0$, we have
\[
{\mathcal H}_\alpha^\eta= \eta^{-1} \Prob \Bigl(\sup
_{k\in\Z} Z_{k\eta
} = 0 \Bigr)
\]
and therefore
\[
{\mathcal H}_\alpha= \lim_{\eta\downarrow0} \eta^{-1}
\Prob \Bigl(\sup_{k\in\Z} Z_{k\eta} = 0 \Bigr).
\]
\end{proposition}

From the point of view of simulation, one difficulty with this
representation is that
one would have to estimate small probabilities when $\eta$ is small.
Unless one develops special techniques, it would require many
simulation replications
to reliably estimate these probabilities. As discussed below, such a
task is computationally
extremely intensive.

\section{An auxiliary bound}
\label{sec:auxiliary}
This section presents a simple auxiliary bound which plays a key role in
the next section.
To formulate it, let $Z^\eta_t$ be the following approximation of $Z_t$
on a grid with mesh $\eta>0$:
\[
Z^\eta_t = %
\cases{
Z_{\eta\lfloor t/\eta\rfloor}, &\quad $\mbox{for $t>0$,}$\vspace*{2pt}
\cr
Z_{\eta\lceil t/\eta\rceil}, &\quad $
\mbox{otherwise,}$} %
\]
and define $B^\eta_t$ similarly in terms of $B_t$.

Let $J$ be a fixed compact closed interval, assumed to be fixed
throughout this section.
We write
\[
\Delta(\eta) = \sup_{t\in J} \bigl(Z_t -
Z^\eta_t\bigr),\qquad \delta(\eta) = \sqrt2 \sup
_{t\in J} \bigl(B_t - B^\eta_t
\bigr).
\]
Define $M_J = \sup_{u\in J} \mathrm{e}^{Z_u}$ and $S^\eta_J = \int_J \mathrm{e}^{Z^\eta_u}\,\mathrm{d}u$.
Note that
\[
\frac{M_J}{S_J^\eta} \le \mathrm{e}^{\Delta(\eta)}\frac{M_J^\eta}{S_J^\eta} \le \frac1\eta
\mathrm{e}^{\Delta(\eta)}.
\]
Given an event $E$, we have for $\tau>\mathrm{e}^{\Expec\Delta(\eta)}$,
%
\begin{eqnarray*}
&&\Expec\bigl(M_J/S_J^\eta;E\bigr)
\\
&&\quad\le \Expec\bigl(M_J/S_J^\eta
;M_J/S_J^\eta>\tau/\eta\bigr) +
\frac{\tau}{\eta} \Prob(E)
\\
&&\quad= \frac1\eta\int_{\tau}^\infty\Prob
\bigl(M_J/S_J^\eta>y/\eta\bigr) \,\mathrm{d}y +
\frac{\tau}{\eta} \Prob\bigl(M_J/S_J^\eta>
\tau/\eta\bigr) + \frac{\tau}\eta \Prob(E)
\\
&&\quad\le\frac1\eta\int_{\tau}^\infty\Prob
\bigl(\mathrm{e}^{\Delta(\eta)}>y\bigr) \,\mathrm{d}y + \frac{\tau}{\eta} \Prob\bigl(\mathrm{e}^{\Delta(\eta)}>
\tau\bigr) + \frac{\tau}\eta \Prob(E)
\\
&&\quad\le \frac1\eta\int_{\tau}^\infty\exp \biggl(-
\frac{(\log(y)-\Expec
\Delta(\eta))^2}{4\eta^\alpha} \biggr)\,\mathrm{d}y + \frac{\tau}{\eta} \exp \biggl(-\frac{(\log(\tau)-\Expec\Delta(\eta))^2}{4\eta^\alpha}
\biggr) + \frac{\tau}\eta \Prob(E),
\end{eqnarray*}
where the last inequality uses Borell's inequality, for example, Adler and
Taylor \cite{adlertaylor:randomfields2007}, Theorem~2.1.1.
We can bound this further by bounding $\Expec\Delta(\eta)$.
After setting
\[
\kappa(\eta) = \sup_{t\in J} \bigl(\Var(Z_t) - \Var
\bigl(Z_t^\eta\bigr)\bigr),
\]
we obtain that $\Delta(\eta) \le\kappa(\eta) + \delta(\eta)$.
We next want to apply Theorem~1.3.3 of Adler and
Taylor~\cite{adlertaylor:randomfields2007} to bound $\Expec\delta(\eta
)$, but
the statement of this theorem contains an unspecified constant.
Our numerical experiments require
that all constants be explicit, and therefore we directly work with the
bound derived in the proof of this theorem.
Choose $r=1/(2\eta^{\alpha/2})$, and set $N_j = |J|r^{j/H}$. The proof
of this theorem shows that
\[
\Expec\delta(\eta) \le \sqrt{\frac{2\pi}{\log(2)}} \sum
_{j=2}^\infty2^{3/2} r^{-j+1} \sqrt
{\log\bigl(2^{j+1} N_j^2\bigr)}=:
\mathcal E(\eta),
\]
which is readily evaluated numerically.

As a result, whenever $\tau>\mathrm{e}^{\mathcal E(\eta)+\kappa(\eta)}$, we have
\begin{eqnarray*}
\Expec\bigl(M_J/S_J^\eta;E\bigr)&\le& \frac1
\eta\int_{\tau}^\infty\exp \biggl(-\frac{(\log(y)-\kappa(\eta
)-\mathcal E(\eta))^2}{4\eta^\alpha}
\biggr)\,\mathrm{d}y
\\
&&{}+ \frac{\tau}{\eta} \exp \biggl(-\frac{(\log(\tau)-\kappa(\eta)-\mathcal
E(\eta))^2}{4\eta^\alpha} \biggr) +
\frac{\tau}\eta\Prob(E).
\end{eqnarray*}
To apply this bound, one needs to select $\tau$ appropriately. Note
that we may let $\tau$ depend on the interval $J$.

\section{Estimation}
\label{sec:estimation}
This section studies the effect of truncation and discretization of $Z$
on $\mathcal H_\alpha$.
The bounds we develop are used in the next section, where we
perform a simulation study in order to estimate~$\mathcal H_\alpha$.

In addition to $\mathcal H_\alpha$ and $\mathcal H_\alpha^\eta$, the
following quantities
play a key role throughout the remainder of this paper:
\[
\mathcal H_\alpha(T) = \Expec \biggl[\frac{\sup_{-T\le t \le
T}\mathrm{e}^{Z_t}}{\int_{-T}^T
\mathrm{e}^{Z_t}\,\mathrm{d}t} \biggr],\qquad
\mathcal H_\alpha^\eta(T) = \Expec \biggl[\frac{\sup_{-T/\eta\le k\le T/\eta}\mathrm{e}^{Z_{k\eta}}}{\eta\sum_{-T/\eta\le k\le T/\eta} \mathrm{e}^{Z_{k\eta}}}
\biggr],
\]
where it is implicit that $t$ is a continuous-time parameter and $k$
only takes integer values.
Throughout, we assume that the truncation horizon $T>0$ and mesh size
$\eta$ are fixed.
We also assume for convenience that $T$ is an integer multiple of $\eta$.

We now introduce some convenient abbreviations.
For fixed $0<a_1< a_2<\cdots$, we write $J_0=(-a_1,a_1)$ and $J_j =
J_j^+ =
[a_j,a_{j+1})$,
$J_{-j} = J_j^- = (-a_{j+1},-a_j]$ with $j\ge1$. Throughout this
section, we use $a_1=T$.
Write $M_j = \sup_{t\in J_j} \mathrm{e}^{Z_t}$, $S_j = \int_{J_j} \mathrm{e}^{Z_t} \,\mathrm{d}t$,
$M_j^\eta= \sup_{k:k\eta\in J_j} \mathrm{e}^{Z_{k\eta}}$, and $S^\eta_j = \eta
\sum_{k:k\eta\in J_j} \mathrm{e}^{Z_{k\eta}}$,
and set $M = \sup_{j\in\Z} M_j$, $S=\sum_{j\in\Z} S_j$, and $S^\eta=
\sum_{j\in\Z} S^\eta_j$.
The length of an interval $J_j$ is denoted by $|J_j|$.

The first step in our error analysis is a detailed comparison of
$\mathcal H_\alpha=\Expec(M/S^\eta)$ and $\Expec(M_0/S_0^\eta)$, which
entails truncation
of the horizon over which the supremum and sum are taken.
As a second step, we compare $\Expec(M_0/S_0^\eta)$ to $\mathcal
H_\alpha^\eta(T)=\Expec(M_0^\eta/S_0^\eta)$,
which entails approximating the maximum on a discrete mesh.

\subsection{Truncation}
This subsection derives upper and lower bounds on $\Expec(M/S^\eta)$ in
terms of $\Expec(M_0/S_0^\eta)$.
For convenience we derive our error bounds for $a_j=T (1+\gamma)^{j-1}$
for $j\ge1$,
for some $\gamma>0$.
Presumably sharper error bounds can be given when the choice of the
$a_j$ is optimized.

\subsubsection{An upper bound}
We derive an upper bound on $\Expec(M/S^\eta)$ in terms of $\Expec
(M_0/S_0^\eta)$.
Since $S\ge S_j$ for any $j\in\Z$, we have
%
%
\begin{eqnarray}
\label{eq:summationtruncation} \Expec\bigl(M/S^\eta\bigr) &=& \Expec
\biggl[\frac{M_0}{S^\eta}; M=M_0 \biggr] + \sum
_{j\neq0} \Expec \biggl[\frac{M_j}{S^\eta}; M=M_j
\biggr]
\nonumber
\\
&\leq&\Expec \bigl(M_0/S^\eta_0\bigr) + \sum
_{j\neq0} \Expec \biggl[\frac{M_j}{S^\eta_j};
M_j > 1 \biggr]
\\
&\leq& \Expec\bigl(M_0/S^\eta_0\bigr) + 2\sum
_{j\ge1} \Expec \biggl[\frac{M_j}{S^\eta_j}; \sqrt{2}\sup
_{s\in J_j} B_s > \min_{s\in
J_j}|s|^{\alpha}
\biggr].\nonumber
\end{eqnarray}
Set $E_j=\{\sqrt{2}\max_{s\in J_j} B_s > \min_{s\in J_j}|s|^{\alpha}\}$.
To further bound (\ref{eq:summationtruncation}), we use the bounds
developed in Section~\ref{sec:auxiliary}.
Thus, the next step is to bound $\Prob(E_j)$ from above.
We write $\tau_j$ for $\tau$ used in the $j$th term.
Using the facts that $B$ has stationary increments and is self-similar,
we find
that
by Theorem~2.8 in Adler \cite{adler:gaussianprocesses1990},
%
%
\begin{eqnarray}
\label{eq:expsupbnd} \Expec \Bigl(\max_{s\in J_j} B_s
\Bigr)&=&\Expec \Bigl(\max_{0\le s\le|J_j|} B_s \Bigr)=
|J_j|^{\alpha/2} \Expec \Bigl(\max_{0\le s\le1}
B_s \Bigr)
\nonumber
\\[-8pt]
\\[-8pt]
\nonumber
&\le& 2|J_j|^{\alpha/2} \Expec \Bigl(\max_{0\le s\le1}
sN \Bigr) = |J_j|^{\alpha/2},
\end{eqnarray}
where $N$ stands for a standard normal random variable.
We derive a bound on $\Prob(E_j)$ in a slightly more general form for
later use.
It follows from Borell's inequality that, for $0<a<b$, $c\in\R$,
%
%
\begin{equation}
\label{eq:borellbound} \Prob \Bigl(\sqrt{2}\max_{s\in[a,b]}
B_s > c+ a^{\alpha} \Bigr) \le \exp \biggl(-\frac
{[c+a^\alpha-\sqrt2(b-a)^{\alpha/2}]^2}{4b^\alpha}
\biggr),
\end{equation}
provided $c+ a^\alpha > \sqrt{2} (b-a)^{\alpha/2}$.
Specialized to $\Prob(E_j)$, we obtain that for $j\ge1$, 
$a_j^{\alpha} > \sqrt2 (a_{j+1}-a_j)^{\alpha/2}$,
\[
\Prob(E_j) \leq\exp \biggl\{-\frac{ [a_j^{\alpha} - \sqrt2
(a_{j+1}-a_j)^{\alpha/2}]^2}{4a_{j+1}^{\alpha}} \biggr\} =\exp
\biggl\{-\frac{(a_j^{\alpha/2}-\gamma^{\alpha/2}\sqrt
2)^2}{4(1+\gamma)^\alpha} \biggr\},
\]
provided $T>\gamma2^{1/\alpha}$.

Thus, the error is upper bounded by $\exp(-c' T^\alpha)$ for some
constant $c'$
as $T\to\infty$. As a result, the error decreases to zero much faster
than any
polynomial, unlike the classical representation for which the error
can be expected to be polynomial as previously discussed.
This is one of the key advantages of our new representation.

\subsubsection{A lower bound}
We derive a lower bound on $\Expec(M/S^\eta)$ in terms of $\Expec
(M_0/S_0^\eta)$ as follows:
\begin{eqnarray*}
\Expec\bigl(M/S^\eta\bigr) &\geq&\Expec \biggl[\frac{M_0}{S^\eta_0}\cdot
\frac
{S^\eta_0}{S^\eta_0 +
\sum_{j\ne0} S^\eta_j}; \varepsilon S^\eta_0 \ge\sum
_{j\neq0} S^\eta _j \biggr]
\\
&\geq&\frac1{1+\varepsilon} \Expec \biggl[\frac{M_0}{S^\eta_0}; \varepsilon
S^\eta_0 \ge\sum_{j\neq0}
S^\eta_j \biggr]
\\
&=& \frac{1}{1+\varepsilon}\Expec\bigl(M_0/S^\eta_0
\bigr)- \frac{1}{1+\varepsilon}\Expec \biggl[\frac{M_0}{S^\eta_0}; \varepsilon
S^\eta_0 < \sum_{j\neq0}
S^\eta_j \biggr].
\end{eqnarray*}

Set $E =  \{\varepsilon S^\eta_0 < \sum_{j\neq0} S^\eta_j \}$. To
apply the technique from Section~\ref{sec:auxiliary}, we
seek an upper bound on $\Prob(E)$. Let $0<\delta<T$, to be determined later.
Since $S^\eta_0 \geq\eta$, we obtain
\[
\Prob(E) \leq\Prob \biggl(\sum_{j\ne0}
S^\eta_j>\varepsilon\eta \biggr)
\nonumber
\leq 2\sum
_{j\ge1} \Prob \bigl(S^\eta_j>
\varepsilon\eta q_j \bigr)
\]
for any probability distribution $\{q_j\dvt j\neq0\}$.
We find it convenient to take $q_j = \psi(1+\psi)^{-|j|}/2$ for some
$\psi>0$ and $j\neq0$.
An upper bound on $S^\eta_j$ for $j\ge1$ is
\[
S^\eta_j \leq(a_{j+1} - a_j)\mathrm{e}^{-a_j^{\alpha}}\mathrm{e}^{\sqrt{2} \max_{s \in J_j}
B_s}=
\gamma a_j \mathrm{e}^{-a_j^{\alpha}}\mathrm{e}^{\sqrt{2} \max_{s \in J_j}
B_s}.
\]
For $j\ge1$, we therefore have
\begin{eqnarray*}
\Prob \bigl(S^\eta_j>\varepsilon\eta q_j
\bigr) &\le& \Prob \bigl(\mathrm{e}^{\sqrt{2} \max_{s \in J_j}B_s}>\varepsilon\eta \mathrm{e}^{a_j^{\alpha}}
q_j/(\gamma a_j) \bigr)
\\[-2pt]
&=& \Prob \Bigl({\sqrt{2} \max_{s \in J_j}B_s}>a_j^{\alpha}
+\log \bigl[\varepsilon\eta q_j/(\gamma a_j) \bigr]
\Bigr)
\\[-2pt]
&\le& \exp \biggl(-\frac{ (\log [\varepsilon\eta
q_j/(\gamma a_j) ] + a^\alpha_j
-\sqrt2 \gamma^{\alpha/2} a_j^{\alpha/2} )^2}{4(1+\gamma)^\alpha
a_{j}^\alpha} \biggr),
\end{eqnarray*}
provided $T$ is large enough so that the expression inside the square
is nonnegative.
The last inequality follows from (\ref{eq:borellbound}).

\subsection{Approximating the supremum on a mesh}
We now find upper and lower bounds on $\Expec(M_0/S_0^\eta)$ in terms
of $\Expec(M_0^\eta/S_0^\eta)$.


For the upper bound, we note that
\[
\Expec\bigl(M_0/S^\eta_0\bigr)\le
\mathrm{e}^{\varepsilon}\Expec\bigl(M^\eta_0/ S^\eta_0
\bigr) +\Expec\bigl(M_0/S^\eta_0;
\Delta_0(\eta) > \varepsilon\bigr).
\]
We use the technique from Section~\ref{sec:auxiliary} to bound $\Expec
(M^\eta_0/ S^\eta_0; \Delta_0(\eta)>\varepsilon)$, which requires a bound on
$\Prob(\Delta_0(\eta)>\varepsilon)$.
Writing $\kappa_0(\eta) = \max(\eta^\alpha,T^\alpha-(T-\eta)^\alpha)$,
we use the self-similarity in conjunction with Borell's inequality and
(\ref{eq:expsupbnd}) to deduce that
\begin{eqnarray*}
\Prob \bigl(\Delta_0(\eta)>\varepsilon \bigr) &\le& \Prob \Bigl(
\sup_{t\in(-T,T)} \sqrt2 \bigl(B_t- B^\eta_t
\bigr)> \varepsilon-\kappa_0(\eta) \Bigr)
\\[-2pt]
&\le& \frac{2T}{\eta} \Prob \Bigl(\sqrt2 \sup_{t\in(0,1)}
\eta^{\alpha
/2} B_t > \varepsilon-\kappa_0(\eta)
\Bigr)
\\[-2pt]
&\le& \frac{2T}{\eta} \Prob \biggl(\sup_{t\in[0,1]}
B_t > \frac{\varepsilon-\kappa_0(\eta)}{\sqrt2 \eta^{\alpha/2}} \biggr)
\\[-2pt]
&\le& \frac{2T}{\eta} \exp \biggl(-\frac12 \biggl[\frac{\varepsilon-\kappa
_0(\eta)}{\sqrt2\eta^{\alpha/2}}-1
\biggr]^2 \biggr),
\end{eqnarray*}
provided $\varepsilon>\kappa_0(\eta)$.\vadjust{\goodbreak}

A lower bound on $\Expec(M_0/ S^\eta_0)$ in terms of $\Expec(M^\eta
_0/S^\eta_0)$ follows trivially:
\[
\Expec\bigl(M_0/ S^\eta_0\bigr) \ge\Expec
\bigl(M^\eta_0/S^\eta_0\bigr).\vspace*{-2pt}
\]

\subsection{Conclusions}
We summarize the bounds we have obtained.
For any $\varepsilon>0$, we have derived the following upper bound:
%
%
\begin{equation}
\label{eq:UBerroranalysis} \mathcal H_\alpha\le \mathrm{e}^{\varepsilon}\Expec
\bigl(M^\eta_0/ S^\eta_0\bigr) +
\Expec \biggl[\frac{M_0}{S^\eta_0};\Delta_0(\eta) > \varepsilon \biggr]
+ 2\sum_{j\ge1} \Expec \biggl[\frac{M_j}{S^\eta_j};
\sqrt{2}\sup_{s\in J_j} B_s > \min_{s\in J_j}|s|^{\alpha}
\biggr],
\end{equation}
where the second and third terms are bounded further using Section~\ref{sec:auxiliary}.
Note that this requires selecting a $\tau$ for each of the terms; we
will come back
to this in the next section.

For any $\varepsilon>0$, we have derived the following lower bound:
%
%
\begin{equation}
\label{eq:LBerroranalysis} \mathcal H_\alpha\ge \frac{1}{1+\varepsilon}\Expec
\bigl(M^\eta_0/ S^\eta_0\bigr) -
\frac{1}{1+\varepsilon}\Expec \biggl[\frac{M_0}{S^\eta_0}; \varepsilon
S^\eta_0 < \sum_{j\neq0}
S^\eta_j \biggr],
\end{equation}
and we again use Section~\ref{sec:auxiliary}.
We note that we may choose a different $\varepsilon$ for the upper bound
and the lower bound,
which we find useful in the next section.\vspace*{-2pt}

\section{Numerical experiments}\vspace*{-2pt}
\label{sec:numexp}
This section consists of two parts.
The first part studies $\mathcal H_\alpha^\eta(T)$
for suitable choices of the simulation horizon $T$ and the
discretization mesh $\eta$,
and uses the previous section to estimate bounds on $\mathcal H_\alpha$.
In the second part of this section, we present a heuristic method for
obtaining sharper estimates
for $\mathcal H_\alpha$.

Simulation of fractional Brownian motion is highly nontrivial,
but there exists a vast body of literature on the topic.
The fastest available algorithms simulate the process on an equispaced grid,
by simulation of the
(stationary) increment process, which often called fractional Gaussian noise.
We use the method of Davies and
Harte \cite{daviesharte:hursteffect1987}
for simulating $n$ points of a fractional Gaussian noise.
This method requires that $n$ be a power of two.
In this approach, the covariance matrix is embedded in a so-called
circulant matrix, for which the eigenvalues can easily be computed.
The algorithm relies on the Fast Fourier Transform (FFT) for maximum efficiency;
the computational effort is of order $n\log n$ for a sample size of
length $n$.
For more details on simulation of fractional Brownian motion, we refer to
Dieker \cite{dieker:mastersthesis2002}.\vspace*{-2pt}

\subsection{Confidence intervals}
Our next aim is to give a point estimate for $\mathcal H_\alpha^\eta(T)$
and use the upper and lower bounds from the previous section to
obtain an interval estimate for Pickands' constant $\mathcal H_\alpha$.

The truncation and discretization errors both critically depend on
$\alpha$, but we choose $T$ and $\eta$ to be fixed throughout our
experiments\vadjust{\goodbreak}
in order to use a simulation technique known as \textit{common random numbers}.
This means that the same stream of (pseudo)random numbers is used for
all values of $\alpha$. By choosing $T$ and $\eta$ independent of
$\alpha$,
the realizations of fractional Brownian motion in the $n$th simulation
replication
are perfectly dependent for different values of $\alpha$.
As a result, our estimate of $\mathcal H_\alpha$ as a function of
$\alpha$ is smoothened
without any statistical sacrifice.

Since $T$ and $\eta$ are fixed, our estimates for $\mathcal H_\alpha$ are
likely to be far off from $\mathcal H_\alpha^\eta(T)$ for small~$\alpha$.
In that regime our algorithm becomes unreliable, since the truncation
horizon would have to grow
so large that it requires ever more computing power to produce an estimate.
Any method that relies on truncating the simulation horizon suffers
from this problem, and it seems
unlikely that truncation can be avoided.
There is some understanding of the asymptotic behavior of $\mathcal
H_\alpha$ as $\alpha\downarrow0$
(Shao \cite{shao:bounds1996},
Harper \cite{harper2010}) so this regime is arguably less
interesting from a simulation point of view.
Since we cannot trust the simulation output for small $\alpha$, we
focus our experiments on $\alpha\ge7/10$.

Somewhat arbitrarily, we chose to calibrate errors using $\alpha=1$,
so that our estimates of $\mathcal H_\alpha(T)$ are close to $\mathcal
H_\alpha$ for $\alpha\ge1$.
The closer one sets the calibration point to $0$,
the higher one has to choose $T$ (and thus more computing power).
We estimate $\mathcal H_\alpha^\eta(T)$ using $1500$ simulation replications,
which takes about three days on a modern computer for each value of
$\alpha$.
We carry out the simulation for $\alpha=14/20, 15/20, \ldots,40/20$,
and interpolate
linearly between the simulated points.
A high-performance computing environment is used to run the experiments
in parallel.

We choose the parameters so that the simulated error bounds
from the previous section yield an error of approximately $3\%$ for
$\alpha=1$.
The most crucial parameter in the error analysis is $\varepsilon$.
We note that a different $\varepsilon$ can be used for the lower and upper
bound, and
that $\varepsilon$ may depend on $\alpha$, so we take advantage of this
extra flexibility to
carefully select $\varepsilon$.
For the upper bound in (\ref{eq:UBerroranalysis}) we use $\varepsilon=
0.005+ 0.025 \cdot(2-\alpha)$, and
for the lower bound in (\ref{eq:LBerroranalysis}) we use $\varepsilon
=(0.005+ 0.025\cdot(2-\alpha)) /3$.
We use $T=128$ and $\eta=1/2^{18}$.

We next discuss how we have chosen the other parameters in the error
analysis from Section~\ref{sec:estimation}.
These have been somewhat optimized.
Equations~(\ref{eq:UBerroranalysis}) and (\ref{eq:LBerroranalysis})
produce bounds on $\mathcal H_\alpha$ in terms of $\mathcal H^\eta
_\alpha(T)$
in view of Section~\ref{sec:auxiliary}, but this requires selecting
some $\tau$ for each term for which Section~\ref{sec:auxiliary}
is applied. We use $\tau_j = 1.3\cdot(1.005)^{j-1}$ for the $j$-term
in the infinite sum, and $\tau=1.4$ for any of the other terms.
We set $\gamma=0.025$ for the growth rate of $a_j$, and we use $\psi=
0.3$ for the decay rate of $q_j$.
For these parameter values, all event-independent terms in Section~\ref{sec:auxiliary} are negligible.
Finally, we replace $\mathcal H^\eta_\alpha(T)$ in the resulting bounds
with its estimate.

In Figure~\ref{fig1}, we plot our estimates of $\mathcal H_\alpha^\eta(T)$
as a function of $\alpha$ (blue, solid), along with their 95\%
confidence interval (green, dotted)
and our bounds for $\mathcal H_\alpha$ (red, dash-dotted). The
numerical values are given in Appendix~\ref{sec:simulatedvalues}.
Note that the errors we find for $\alpha<1$
are so large that our error bounds are essentially useless. We do
believe that the simulated
values are reliable approximations to $\mathcal H_\alpha$, but the
bounds from our error analysis
are too loose.
%

\begin{figure}

\includegraphics{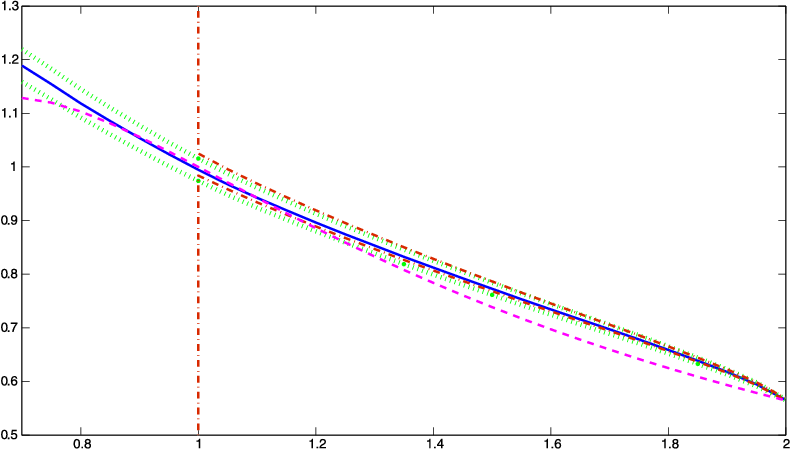}

\caption{Point estimates (blue, solid) and interval estimates (green,
dotted) for $\mathcal H_\alpha^\eta(T)$ as a function of~$\alpha$.
Our error analysis shows that $|\mathcal H_\alpha^\eta(T)-\mathcal
H_\alpha|$ is
at most $0.03$ for $\alpha\ge1$ (red, dash-dotted). We also plot
$1/\Gamma(1/\alpha)$ (magenta, dashed).}\vspace*{-3pt}
\label{fig1}
\end{figure}

A well-known conjecture states that $\mathcal H_\alpha=1/\Gamma(1/\alpha
)$ (D{\c{e}}bicki and Mandjes \cite{MR2834197}),
but (to our knowledge) it lacks any foundation other than that
$\lim_{\alpha\downarrow0} \mathcal H_\alpha= \lim_{\alpha\downarrow0}
1/\Gamma(1/\alpha)=0$,
$\mathcal H_1 = 1/\Gamma(1)$, and $\mathcal H_2=1/\Gamma(1/2)$.
A referee communicated to us that this conjecture is due to K.~Breitung.
Our simulation gives strong evidence that this conjecture is not
correct: the function $1/\Gamma(1/\alpha)$ is
the magenta, dashed curve in Figure~\ref{fig1}, and we see that
the confidence interval and error bounds are well above the curve for
$\alpha$ in the range 1.6--1.8.
Note that we cannot \textit{exclude} that this conjecture holds, since our
error bounds are based on Monte
Carlo experiments.
However, this formula arguably serves as a reasonable approximation for
$\alpha\ge1$.\vspace*{-2pt}

\subsection{A regression-based approach}\vspace*{-2pt}
In the previous subsection, we approximated $\mathcal H_\alpha$ by
$\mathcal H_\alpha^\eta(T)$.
The main contribution to the error is the discretization step,
so we now focus on a refined approximation based on the behavior of
$\mathcal H_\alpha^\eta(T)$
as $\eta\downarrow0$.

This approach relies on the rate at which $\mathcal H_\alpha^\eta(T)$
converges to $\mathcal H_\alpha(T)$.
We state this as a conjecture, it is
outside the scope of the current paper to (attempt to) prove it.

\begin{conjecture}
\label{conj}
For fixed $T>0$, we have
$\lim_{\eta\downarrow0} \eta^{-\alpha/2}[\mathcal{H}_\alpha(T) -
\mathcal{H}^\eta_\alpha(T)] \in(0,\infty)$.
We also have $\lim_{\eta\downarrow0} \eta^{-\alpha/2} [\mathcal
{H}_\alpha- \mathcal{H}^\eta_\alpha] \in(0,\infty)$.
\end{conjecture}

We motivate this conjecture as follows. We focus on the last part of
the conjecture for brevity.
Since $\mathcal{H}_\alpha= \Expec[M/S^\eta]$ by Proposition~\ref
{prop:discreteanalog}, we obtain that
\[
\eta^{-\alpha/2}\bigl[\mathcal{H}_\alpha- \mathcal{H}^\eta_\alpha
\bigr] = \Expec \biggl[\eta^{-\alpha/2}\bigl(\mathrm{e}^{m-m^\eta} -1\bigr) \times
\frac{M^\eta}{S^\eta} \biggr],\vadjust{\goodbreak}
\]
where $m^\eta=\log M^\eta$ and $m=\log M$.
The right-hand side equals approximately
\[
\Expec \biggl[\eta^{-\alpha/2}\bigl(m-m^\eta\bigr)\times
\frac{M^\eta}{S^\eta} \biggr].
\]
This expectation involves a product of two random variables. The random
variable $M^\eta/S^\eta$ converges
almost surely to the finite random variable $M/S$ as $\eta\downarrow0$.
Although we are not aware of any existing results on the behavior of
$\eta^{-\alpha/2} (m-m^\eta)$ or its expectation,
we expect that the random variable $\eta^{-\alpha/2} (m-m^\eta)$
converges in distribution.
Indeed, this is suggested by prior work on related problems,
see Asmussen \textit{et~al.} \cite{MR1384357} for the case $\alpha=1$
and H{\"u}sler \textit{et~al.} \cite{husler:uniformnorms2003},
Seleznjev~\cite{MR1387887} for general results on interpolation approximations
for Gaussian processes
(which is different but related).
The rate of convergence of $m^\eta$ to $m$ (or for finite-horizon
analogs) seems to be of
general fundamental interest, but falls outside the scope of this paper.

Conjecture~\ref{conj} implies that for some $c=c(T)>0$, for small $\eta
$, we have approximately
\[
\mathcal H_\alpha^\eta(T) = \mathcal H_\alpha(T) -c
\eta^{\alpha/2}.
\]
This allows us to perform an ordinary linear regression to
simultaneously estimate $c$ and $\mathcal H_\alpha(T)$
from (noisy) estimates of $\mathcal H_\alpha^\eta(T)$ for different
(small) values of $\eta$ and fixed $\alpha$.
One could use the same simulated fractional Brownian motion trace for
different values of $\eta$,
but it is also possible to use independent simulation experiments for
different values of $\eta$.
The latter approach is computationally less efficient, but it has the
advantage that
classical regression theory becomes available for constructing
confidence intervals of~$\mathcal H_\alpha(T)$.\looseness=-1

Even though we do not have a formal justification for this approach,
we have carried out regressions with the same simulated trace for
different values of $\eta$.
The results are reported in Figure~\ref{fig2}.
The simulation experiments are exactly the same as those underlying
Figure~\ref{fig1},
and in particular we have used the same parameter values.
The red, dashed curves are estimates for $\mathcal H_\alpha^\eta(T)$
for $\eta=2^{-14},2^{-13},2^{-12},2^{-11}$.
Using the regression approach, we estimate $\mathcal H_\alpha^\eta(T)
$ for $\eta=2^{-18}$ and
compare it with our simulation estimates for the same value of $\eta$
(blue, solid). The two resulting curves are indistinguishable
in Figure~\ref{fig2}, and the difference is of order $10^{-3}$.
We have also plotted our regression-based estimate of $\mathcal H_\alpha
(T)$ (green, dash-dotted).

\begin{figure}

\includegraphics{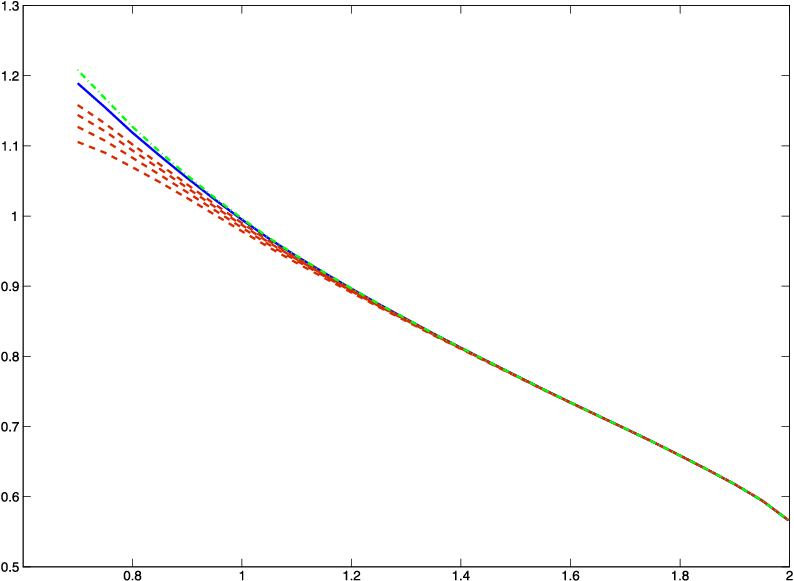}

\caption{Estimation of $\mathcal H_\alpha^\eta(T)$ for different values
of $\eta$.}
\label{fig2}
\end{figure}

It is instructive to look at the resulting estimate for $\mathcal
H_1(T)$, since we know that $\mathcal H_1=1$.
Our estimate for $\mathcal H_1(T)$ is $0.9962650$, which is indeed
closer to its true value.
As the number of simulation replications increases, we expect much more
improvement.\vspace*{-3pt}

\begin{appendix}
\section{Proofs}\vspace*{-3pt}
\label{appendix}

\begin{pf*}{Proof of Proposition~\ref{prop:UIMT}}
Our proof of (\ref{eq:unifintMT}) uses several ideas that are similar
to those in Sections~\ref{sec:auxiliary} and \ref{sec:estimation},
so our exposition here is concise.
We fix some $\eta$ for which $1/\eta$ is a (large) integer; its exact
value is irrelevant.
Recall that the quantities $J_j, M_j, S_j, M_j^\eta, S_j^\eta$ from
Section~\ref{sec:estimation}
have been introduced with respect to parameters $0<a_1<a_2<\cdots$.
Here we
use different choices: $a_1 = \lceil2^{1/(2\alpha)}\rceil$, $a_j =
a_{j-1} + 1$ for $j\ge2$.\vadjust{\goodbreak}

Abusing notation slightly, we write $M_{[-uT,(1-u)T]} = \sup_{s\in
[-uT,(1-u)T]} \mathrm{e}^{Z_{s}}$ and\break  $S_{[-uT,(1-u)T]} = \int_{-uT}^{(1-u)T}
\mathrm{e}^{Z_{s}} \,\mathrm{d}s$.
Since $M_{[-uT,(1-u)T]}\to M$ and $S_{[-uT,(1-u)T]}\to S$ almost surely
as $T\to\infty$ for $u\in(0,1)$,
both claims follow after showing that
%
%
\renewcommand{\theequation}{\arabic{equation}}
\setcounter{equation}{8}
\begin{equation}
\label{eq:unifintMT} \lim_{A\to\infty} \sup_{T>0}
\sup_{u\in(0,1)} \Expec \biggl[\frac
{M_{[-uT,(1-u)T]}}{S_{[-uT,(1-u)T]}};
\frac{M_{[-uT,(1-u)T]}}{S_{[-uT,(1-u)T]}}>A \biggr]=0.
\end{equation}

Write $\kappa_j =\kappa_j(\eta)= \sup_{t\in J_j} [\Var(Z_t) - \Var
(Z_t^\eta)]$.
First, suppose that $-uT$ and $(1-u)T$ lie in $\{\ldots
,-a_2,-a_1,a_1,a_2,\ldots\}$.
On the event $\{M_{[-uT,(1-u)T]} = M_j\}$ for some $j\in\Z$, we have
%
%
\begin{eqnarray}
\label{eq:boundMSprop1} \frac{M_{[-uT,(1-u)T]}}{S_{[-uT,(1-u)T]}} &\le& \frac{M_j}{S_j} \le
\mathrm{e}^{2\sqrt2 \sup_{s\in J_j} |B_s-B_s^\eta| +\kappa_j} \frac{M^\eta
_j}{S^\eta_j}
\nonumber
\\[-8pt]
\\[-8pt]
\nonumber
&\le& \frac1\eta \mathrm{e}^{2\sqrt2 \sup_{s\in J_j} |B_s-B_s^\eta| +\kappa
_j}.
\end{eqnarray}
Note that this bound remains valid if $-uT$ and $(1-u)T$ fail to lie in
$\{\ldots,-a_2,-a_1,\break a_1,a_2,\ldots\}$.

Since $\Expec [\mathrm{e}^{2 \sqrt2\sup_{s\in J_0} |B_s-B^\eta_s|} ]
<\infty$ by Borell's inequality,
(\ref{eq:unifintMT}) follows after we establish that
\[
\lim_{A\to\infty} \sum_{j=1}^\infty
\Expec \bigl[\mathrm{e}^{2\sqrt2 \sup_{s\in
J_j} |B_s-B_s^\eta| +\kappa_j}; \mathrm{e}^{2\sqrt2 \sup_{s\in J_j}
|B_s-B_s^\eta| +\kappa_j}>A, M_j>1 \bigr]=0.
\]
To this end, we observe that for $j\ge1$
\begin{eqnarray*}
&& \Expec \bigl[\mathrm{e}^{2\sqrt2 \sup_{s\in J_j} |B_s-B_s^\eta| +\kappa_j}; M_j>1 \bigr]
\\
&&\quad\le \sqrt{\Expec \bigl[\mathrm{e}^{4\sqrt2 \sup_{s\in J_j} |B_s-B_s^\eta| +2\kappa
_j} \bigr] \Prob(M_j>1)}
\\
&&\quad\le\sqrt{ \mathrm{e}^{2\kappa_j }\Expec \bigl[\mathrm{e}^{4\sqrt2 \sup_{s\in[0,1]}
|B_s-B_s^\eta|} \bigr] \Prob
\Bigl(\sup_{t\in J_j} \sqrt2B_t >a_j^\alpha
\Bigr)}
\\
&&\quad\le C\mathrm{e}^{\kappa_j} \exp \biggl(-\frac{(a_j^\alpha-\sqrt2\Expec[\sup_{t\in J_j} B_t])^2}{8(a_j+1)^\alpha} \biggr)
\\
&&\quad\le C\mathrm{e}^{\kappa_j} \exp \biggl(-\frac{(a_j^\alpha-\sqrt
2])^2}{8(a_j+1)^\alpha} \biggr),
\end{eqnarray*}
where $C$ denotes some constant and we have used (\ref{eq:expsupbnd})
to obtain the last inequality. Note that $a_j^\alpha> \sqrt2$
for our choice of $a_j$.
The resulting expression is summable, which establishes the required
inequality by the monotone convergence theorem.
\end{pf*}

\begin{pf*}{Proof of Proposition~\ref{prop:discreteanalog}}
Our starting point is (\ref{eq:reprcountingmeasure}) and the
accompanying remarks.
By Theorem~1.5.5 in Billingsley \cite{billingsley:convprob1968}, it suffices to
show that $\textrm{Leb}(E) = 0$, where
$E$ consists of all $u\in[0,1]$ for which
\[
\lim_{T\to\infty} \Expec \biggl[\frac{M_{[-u_T T,(1-u_T) T]}} {S^\eta
_{[-u_T T,(1-u_T) T]}} \biggr] = \Expec
\biggl[\frac M{S^\eta} \biggr]
\]
fails to hold for some $\{u_T\}$ with $u_T\to u$.
With minor modifications to the bound (\ref{eq:boundMSprop1}) since, we
work with $S^\eta$ instead of $S$,
the proof of Proposition~\ref{prop:UIMT} shows that
\[
\lim_{A\to\infty} \sup_{T>0} \sup
_{u\in(0,1)} \Expec \biggl[\frac
{M_{[-uT,(1-u)T]}}{S^\eta_{[-uT,(1-u)T]}}; \frac{M_{[-uT,(1-u)T]}}{S^\eta_{[-uT,(1-u)T]}}>A
\biggr]=0.
\]
This implies that $E\subseteq\{0,1\}$, so its Lebesgue measure is zero.
\end{pf*}

\begin{table}[b]
\def\arraystretch{0.9}
\caption{Our numerical results}\label{tabB.1}
\begin{tabular*}{\textwidth}{@{\extracolsep{4in minus 4in}}lllll@{}}
\hline
\multicolumn{1}{@{}l}{$\alpha$} & \multicolumn{1}{l}{Estimate $\mathcal H_\alpha^\eta(T)$} & \multicolumn{1}{l}{Sample stddev
$M_0/S_0^\eta$} & \multicolumn{1}{l}{Lower bound $\mathcal H_\alpha$} & \multicolumn{1}{l@{}}{Upper bound
$\mathcal H_\alpha$}\\
\hline
$0.700$ & $1.1888337$ & $0.5998979$ & -- & -- \\
$0.750$ & $1.1543904$ & $0.5614484$ & -- & -- \\
$0.800$ & $1.1184290$ & $0.5257466$ & -- & -- \\
$0.850$ & $1.0855732$ & $0.4919238$ & -- & -- \\
$0.900$ & $1.0539127$ & $0.4625016$ & -- & -- \\
$0.950$ & $1.0235620$ & $0.4360272$ & -- & -- \\
$1.000$ & $0.9946978$ & $0.4116689$ & $0.9837218$ & $1.0250320$ \\
$1.050$ & $0.9674279$ & $0.3892142$ & $0.9582444$ & $0.9956451$ \\
$1.100$ & $0.9424383$ & $0.3665194$ & $0.9338777$ & $0.9687150$ \\
$1.150$ & $0.9191131$ & $0.3442997$ & $0.9111406$ & $0.9435593$ \\
$1.200$ & $0.8963231$ & $0.3239746$ & $0.8889154$ & $0.9190136$ \\
$1.250$ & $0.8743162$ & $0.3048379$ & $0.8674489$ & $0.8953298$ \\
$1.300$ & $0.8532731$ & $0.2864521$ & $0.8469212$ & $0.8726894$ \\
$1.350$ & $0.8322652$ & $0.2698805$ & $0.8264114$ & $0.8501401$ \\
$1.400$ & $0.8121016$ & $0.2540026$ & $0.8067235$ & $0.8285072$ \\
$1.450$ & $0.7922732$ & $0.2390896$ & $0.7873523$ & $0.8072685$ \\
$1.500$ & $0.7727308$ & $0.2248372$ & $0.7682494$ & $0.7863726$ \\
$1.550$ & $0.7531251$ & $0.2112524$ & $0.7490677$ & $0.7654634$ \\
$1.600$ & $0.7342039$ & $0.1970492$ & $0.7305511$ & $0.7453000$ \\
$1.650$ & $0.7155531$ & $0.1821118$ & $0.7122884$ & $0.7254599$ \\
$1.700$ & $0.6970209$ & $0.1665167$ & $0.6941287$ & $0.7057883$ \\
$1.750$ & $0.6782065$ & $0.1503939$ & $0.6756727$ & $0.6858794$ \\
$1.800$ & $0.6585134$ & $0.1339708$ & $0.6563256$ & $0.6651316$ \\
$1.850$ & $0.6384329$ & $0.1156335$ & $0.6365762$ & $0.6440437$ \\
$1.900$ & $0.6176244$ & $0.0953090$ & $0.6160842$ & $0.6222740$ \\
$1.950$ & $0.5944161$ & $0.0698590$ & $0.5931803$ & $0.5981428$ \\
$1.998$ & $0.5663460$ & $0.0146697$ & $0.5653943$ & $0.5692133$ \\
\hline
\end{tabular*}
\end{table}

\begin{pf*}{Proof of Proposition~\ref{prop:albintyperepr}}
As in the proof of Proposition~\ref{prop:discreteanalog}, it suffices
to show that, whenever $u_T\to u\in(0,1)$,
\[
\lim_{T\to\infty} \Prob \Bigl(\sup_{k\in\Z: -u_T T\le k\eta\le(1-u_T)
T}
Z_{k\eta} =0 \Bigr) = \Prob \Bigl(\sup_{k\in\Z}
Z_{k\eta} =0 \Bigr).
\]
A sandwich argument readily establishes that
\[
\lim_{T\to\infty} \sup_{k\in\Z: -u_T T\le k\eta\le(1-u_T) T} Z_{k\eta} =
\sup_{k\in\Z} Z_{k\eta}.
\]
The claim follows since almost sure convergence implies convergence in
distribution.
\end{pf*}

\vspace*{-3pt}\section{Simulated values}\vspace*{-3pt}
\label{sec:simulatedvalues}

This appendix lists our estimates for $\mathcal
H_\alpha^\eta(T)$ in tabular form for $\eta=1/2^{18}$ and $T=128$,
along with the sample standard deviation.
We also list the lower and upper bounds on $\mathcal H_\alpha$,
where we note that these are estimated values since they
depend on $\mathcal H_\alpha^\eta(T)$.
We cannot report these bounds for $\alpha<1$, since our choice of
parameter values causes the methodology to break down.
Our methods can be applied with different parameter values to obtain
bounds in this regime,
but this requires more computing time and is not pursued in this paper. These numerical results are summarized in
Table \ref{tabB.1}.

\end{appendix}

\section*{Acknowledgments}
This work was conducted while B. Yakir was a visiting professor at Georgia Tech,
and was supported in part by Grant 325/09 of the Israeli Science Foundation and Grant
CMMI-1252878 of the National Science Foundation.
We thank Allen Belletti and Lawrence Sharp for technical assistance
with our high-performance computing experiments,
and David Goldberg for valuable discussions.

%



\printhistory

\end{document}